\documentclass[preprint,12pt]{elsarticle}

\usepackage{graphics}

\usepackage{amssymb}
\usepackage{amsthm}

\usepackage[figuresright]{rotating}
\usepackage{subfigure}
\newtheorem{theorem}{Theorem}

\newtheorem{lemma}[theorem]{Lemma}

\usepackage{amssymb}
\usepackage{amsmath}
\usepackage{amsfonts}
\usepackage{listings}
\usepackage{graphicx}
\usepackage{mathptmx}
\usepackage[T1]{fontenc}
\usepackage[scaled]{helvet}
\usepackage{setspace}

\usepackage{graphicx}
\usepackage{xcolor}
\usepackage[nottoc,notlof,notlot]{tocbibind}
\usepackage{booktabs}

\usepackage[pagebackref,linktocpage,breaklinks,colorlinks,%
                linkcolor=black,anchorcolor=black,citecolor=black,%
                filecolor=black,menucolor=black,runcolor=black,%
                urlcolor=black]{hyperref}

\journal{a journal}

\begin{document}

\begin{frontmatter}




\author[label1]{J.C. Jimenez}

\title{Simplified formulas for the mean and variance of linear stochastic
differential equations}

\address[label1]{Instituto de Cibernetica, Matematica y Fisica, Calle 15,
No. 551, entre C y D, Vedado, La Habana, Cuba. e-mail: (see
http://sites.google.com/site/locallinearization/)}

\begin{abstract}
Explicit formulas for the mean and variance of the solutions of
linear stochastic differential equations are derived in terms of an
exponential matrix. This result improved a previous one by means of
which the mean and variance are expressed in terms of a linear
combination of higher dimensional exponential matrices. The
important role of the new formulas for the system identification as
well as numerical algorithms for their practical implementation are
pointed out.
\end{abstract}


\end{frontmatter}

\section{Introduction}

Linear Stochastic Differential Equations (SDEs) define one of the
more simple class of equations frequently used for modeling a
variety of random phenomena. Since long time ago, they have been the
focus of intensive researches resulting in a broad and deep
knowledge of the properties of their solutions. Among these
properties, the mean and variance of the solutions have been well
studied. Specifically, the Ordinary Differential Equations (ODEs)
that describe the dynamics of the mean and variance are well know
(see, e.g., \cite{Arnold 1974,Khasminskii12}). However, since the
explicit solutions of these ODEs were rarely available, numerical
solutions were required during some time. Typically, these
approximate solutions were computed by means of a numerical
integrator for the differential equations or by a numerical
quadrature applied to the integral representation of the mean and
variance (see, e.g., \cite{Mora 2005, Jimenez98 AML}). Later, in
addition to the explicit formulas for the mean and variance of the
scalar linear SDEs and for the multidimensional linear SDEs with
diagonal drift and diffusion coefficients, explicit formulas could
be obtained as well for multidimensional linear SDEs with additive
noise. Indeed, by using the main result of \cite{Van Loan 1978}, the
mean and variance of these equations could be expressed in terms of
exponential matrices. The Padé method proposed in \cite{Van Loan
1978} for the numerical calculation of such exponential matrix is
the more efficient and accurate algorithm currently available for
the computation of the mean and variance of linear SDEs with
additive (see, e.g., \cite{Grewal08,Carbonell06 JCAM}). More
recently, in \cite{Jimenez02 SCL,Jimenez03 IJC}, explicit formulas
for the mean and variance of linear SDEs with multiplicative and/or
additive noises were derived in terms of a linear combination of
seven exponential matrices. The formulas were obtained as solution
of the mentioned ODEs for the mean and variance by using the main
result of \cite{Van Loan 1978} as well. Apart from being of
mathematical interest, these explicit formulas have played a crucial
role in the practical implementation of suboptimal linear filters
\cite{Jimenez02 SCL}, Local Linearization filters \cite{Jimenez03
IJC} and approximate Innovation estimators \cite{Jimenez06 JTSA} for
the identification of continuous-discrete state space models. In a
variety of applications, these methods have shown high effectiveness
and efficiency for the estimation of unobserved components and
unknown parameters of SDEs given a set of discrete observations.
Remarkable is the identification, from actual data, of
neurophysiological, financial and molecular models among others
(see, e.g., \cite{Calderon09,Chiarella09,Jimenez06
APFM,Riera04,Valdes99}). Therefore, a simplification of the formulas
for the mean and variance of linear SDEs might imply a sensible
reduction of the computational cost of the mentioned system
identification methods and, consequently, a positive impact in
applications.

In this paper, simplified explicit formulas for the mean and
variance of linear SDEs are obtained in terms of just one
exponential matrix of lower dimensionality. The formulas are derived
from the solution of the ODEs that describe the evolution of the
mean and the second moment of the SDEs. The variance is then
obtained from the well-known formula that involves the first two
moments. The computational benefits of the simplified formulas are
pointed out.

\section{Notation and Preliminaries}

Let us consider the $d$-dimensional linear stochastic differential equation%
\begin{equation}
d\mathbf{x}(t)=(\mathbf{Ax}(t)+\mathbf{a(}t))dt+\sum\limits_{i=1}%
^{m}(\mathbf{B}_{i}\mathbf{x}(t)+\mathbf{b}_{i}(t))d\mathbf{w}^{i}(t)\text{
}
\label{SS1}%
\end{equation}
for all $t\in\lbrack t_{0},T]$, where $\mathbf{w=(\mathbf{w}}^{1}%
,\ldots,\mathbf{w}^{m}\mathbf{)}$ is an $m$-dimensional standard
Wiener process, $\mathbf{A}$ and $\mathbf{B}_{i}$ are $d\times d$
matrices, and
$\mathbf{a}(t)=\mathbf{a}_{0}+\mathbf{a}_{1}t$ and $\mathbf{b}_{i}%
(t)=\mathbf{b}_{i,0}+\mathbf{b}_{i,1}t$ are $d$-dimensional vectors.
Suppose that there exist the first two moments of $\mathbf{x}$ for
all $t\in\lbrack t_{0},T]$.

The ordinary differential equations for the $d$-dimensional vector
mean $\mathbf{m}_{t}=E(\mathbf{x}(t))$ and the $d\times d$ matrix
second moment
$\mathbf{P}_{t}=E(\mathbf{x}(t)\mathbf{x}^{\intercal}(t))$ of
$\mathbf{x}(t)$ are \cite{Kloeden 1995}
\[
\frac{d\mathbf{m}_{t}}{dt}=\mathbf{Am}_{t}+\mathbf{a}(t)
\]
and%
\[
\frac{d\mathbf{P}_{t}}{dt}=\mathbf{AP}_{t}+\mathbf{P}_{t}\mathbf{A}%
^{\intercal}+\sum\limits_{i=1}^{m}\mathbf{B}_{i}\mathbf{P}_{t}\mathbf{B}%
_{i}^{\intercal}+\mathcal{B}(t),
\]
where
\begin{equation}
\mathcal{B}(t)=\mathbf{a}(t)\mathbf{m}_{t}^{\intercal}+\mathbf{m}%
_{t}\mathbf{a}^{\intercal}(t)+\sum\limits_{i=1}^{m}\mathbf{B}_{i}%
\mathbf{m}_{t}\mathbf{b}_{i}^{\intercal}(t)+\mathbf{b}_{i}\mathbf{m}%
_{t}^{\intercal}\mathbf{B}_{i}^{\intercal}(t)+\mathbf{b}_{i}(t)\mathbf{b}%
_{i}^{\intercal}(t). \label{beta}%
\end{equation}

The solution of these equations can be written as \cite{Jimenez02
SCL,Jimenez05 AMC}
\begin{equation}
\mathbf{m}_{t}=\mathbf{m}_{0}+\mathbf{Le}^{\mathbf{C}(t-t_{0})}\mathbf{r}
\label{solucion m}%
\end{equation}
and
\begin{equation}
vec(\mathbf{P}_{t})=e^{\mathcal{A}(t-t_{0})}(vec(\mathbf{P}_{0})+%
{\displaystyle\int\limits_{0}^{t-t_{0}}}
e^{-\mathcal{A}s}vec(\mathcal{B}(s+t_{0}))ds), \label{solucion P}%
\end{equation}
where $\mathbf{m}_{0}=E(\mathbf{x}(t_{0}))$ and $\mathbf{P}_{0}=E(\mathbf{x}%
(t_{0})\mathbf{x}^{\intercal}(t_{0}))$ are the first two moments of
$\mathbf{x}$ at $t_{0}$, and the matrices $\mathbf{C}$, $\mathbf{L}$
and $\mathbf{r}$ are defined as
\begin{equation}
\mathbf{C}=\left[
\begin{array}
[c]{ccc}%
\mathbf{A} & \mathbf{a}_{1} & \mathbf{Am}_{0}+\mathbf{a}(t_{0})\\
0 & 0 & 1\\
0 & 0 & 0
\end{array}
\right]  \in\Re^{(d+2)\times(d+2)}, \label{Matrix C noauto}%
\end{equation}
$\mathbf{L}=\left[
\begin{array}
[c]{ll}%
\mathbf{I}_{d} & \mathbf{0}_{d\times2}%
\end{array}
\right]  $ and $\mathbf{r}^{\intercal}=\left[
\begin{array}
[c]{ll}%
\mathbf{0}_{1\times(d+1)} & 1
\end{array}
\right]  $ for non-autonomous equations (i.e., with non zero $\mathbf{a}_{1}%
$\textbf{, }$\mathbf{b}_{i,1}$); and as
\begin{equation}
\mathbf{C}=\left[
\begin{array}
[c]{cc}%
\mathbf{A} & \mathbf{Am}_{0}+\mathbf{a}(t_{0})\\
0 & 0
\end{array}
\right]  \in\Re^{(d+1)\times(d+1)}, \label{Matrix  C auto}%
\end{equation}
$\mathbf{L}=\left[
\begin{array}
[c]{ll}%
\mathbf{I}_{d} & \mathbf{0}_{d\times1}%
\end{array}
\right]  $ and $\mathbf{r}^{\intercal}=\left[
\begin{array}
[c]{ll}%
\mathbf{0}_{1\times d} & 1
\end{array}
\right]  $ for autonomous equations (i.e., with
$\mathbf{a}_{1}=$\textbf{ }$\mathbf{b}_{i,1}=\mathbf{0}$). Here,
\begin{equation}
\mathcal{A}=\mathbf{A}\oplus\mathbf{A}+\sum\limits_{i=1}^{m}\mathbf{B}%
_{i}\mathbf{\otimes B}_{i}^{\intercal} \label{Matrix A}%
\end{equation}
is a $d^{2}\times d^{2}$\ matrix, and $\mathbf{I}_{d}$ is the
$d$-dimensional identity matrix. The symbols $vec$, $\oplus$ and
$\otimes$ denote the vectorization operator, the Kronecker sum and
product, respectively.

The following lemma provides simple expressions to computing
multiple integrals involving matrix exponentials such those
appearing in (\ref{solucion P}).

\begin{lemma}
\label{Lemma VanLoan}(\cite{Van Loan 1978}) Let $\mathbf{A}_{1}$,
$\mathbf{A}_{2}$, $\mathbf{A}_{3}$ and $\mathbf{A}_{4}$ be square
matrices, $n_{1}$, $n_{2}$, $n_{3}$ and $n_{4}$ be positive
integers, and set $m$ to be their sum. If the $m\times m$ block
triangular matrix $\mathbf{M}$ is defined by
\[
\mathbf{M}=\left[
\begin{array}
[c]{cccc}%
\mathbf{A}_{1} & \mathbf{B}_{1} & \mathbf{C}_{1} & \mathbf{D}_{1}\\
0 & \mathbf{A}_{2} & \mathbf{B}_{2} & \mathbf{C}_{2}\\
0 & 0 & \mathbf{A}_{3} & \mathbf{B}_{3}\\
0 & 0 & 0 & \mathbf{A}_{4}%
\end{array}
\right]
\begin{array}
[c]{l}%
\}\text{ }n_{1}\\
\}\text{ }n_{2}\\
\}\text{ }n_{3}\\
\}\text{ }n_{4}%
\end{array}
,
\]
then for $s\geq0$%
\[
\left[
\begin{array}
[c]{cccc}%
\mathbf{F}_{1}\mathbf{(}s\mathbf{)} & \text{ }\mathbf{G}_{1}\mathbf{(}%
s\mathbf{)} & \mathbf{H}_{1}\mathbf{(}s\mathbf{)} & \mathbf{K}_{1}(s)\\
0 & \mathbf{F}_{2}\mathbf{(}s\mathbf{)} &
\mathbf{G}_{2}\mathbf{(}s\mathbf{)}
& \mathbf{H}_{2}\mathbf{(}s\mathbf{)}\\
0 & 0 & \mathbf{F}_{3}\mathbf{(}s\mathbf{)} & \mathbf{G}_{3}\mathbf{(}%
s\mathbf{)}\\
0 & 0 & 0 & \mathbf{F}_{4}\mathbf{(}s\mathbf{)}%
\end{array}
\right]  =\exp(s\mathbf{M}),
\]
where
\begin{align*}
\mathbf{F}_{j}(s)  &  \equiv\exp(\mathbf{A}_{j}s),\text{ for }j=1,2,3,4\\
\mathbf{G}_{j}(s)  &  \equiv\int\limits_{0}^{s}\exp(\mathbf{A}_{j}%
(s-u))\mathbf{B}_{j}\exp(\mathbf{A}_{j+1}u)du\text{, for }j=1,2,3\\
\mathbf{H}_{j}(s)  &  \equiv\int\limits_{0}^{s}\exp(\mathbf{A}_{j}%
(s-u))\mathbf{C}_{j}\exp(\mathbf{A}_{j+2}u)du\\
&  +\int\limits_{0}^{s}\int\limits_{0}^{u}\exp(\mathbf{A}_{j}(s-u))\mathbf{B}%
_{j}\exp(\mathbf{A}_{j+1}(u-r))\mathbf{B}_{j+1}\exp(\mathbf{A}_{j+2}r)drdu\\
\mathbf{K}_{1}(s)  &  \equiv\int\limits_{0}^{s}\exp(\mathbf{A}_{1}%
(s-u))\mathbf{D}_{1}\exp(\mathbf{A}_{4}u)du\\
&  +\int\limits_{0}^{s}\int\limits_{0}^{u}\exp(\mathbf{A}_{1}(s-u))[\mathbf{C}%
_{1}\exp(\mathbf{A}_{3}(u-r))\mathbf{B}_{3}+\mathbf{B}_{1}\exp(\mathbf{A}%
_{2}(u-r))\mathbf{C}_{2}]\exp(\mathbf{A}_{4}r)drdu.\\
&  +\int\limits_{0}^{s}\int\limits_{0}^{u}\int\limits_{0}^{r}\exp
(\mathbf{A}_{1}(s-u))\mathbf{B}_{1}\exp(\mathbf{A}_{2}(u-r))\mathbf{B}_{2}%
\exp(\mathbf{A}_{3}(r-w))\mathbf{B}_{3}\exp(\mathbf{A}_{4}w)dwdrdu.
\end{align*}

\end{lemma}

A generalization of the above lemma for integrals with higher
multiplicity is given by the following theorem.

\begin{theorem} (\cite{Carbonell08 JCAM})
\label{Theorem VanLoan} Let $d_{1},d_{2},...,d_{n},$ be positive
integers. If the
$n\times n$ block triangular matrix $\mathbf{A=}[(\mathbf{A}%
_{lj})]_{l,j=1:n} $ is defined by%
\begin{equation*}
\mathbf{A=}%
\begin{pmatrix}
\mathbf{A}_{11} & \mathbf{A}_{12} & ... & \mathbf{A}_{1n} \\
\mathbf{0} & \mathbf{A}_{22} & ... & \mathbf{A}_{2n} \\
\mathbf{0} & \mathbf{0} & \ddots & \vdots \\
\mathbf{0} & \mathbf{0} & \mathbf{0} & \mathbf{A}_{nn}%
\end{pmatrix}%
,
\end{equation*}%
where $(\mathbf{A}_{lj}),l,j=1,...,n$ are $d_{l}\times d_{j}$
matrices such
that $d_{l}=d_{j}$ for $l=j$. Then for $t\geqslant 0$%
\begin{equation*}
e^{\mathbf{A}t}\mathbf{=}%
\begin{pmatrix}
\mathbf{B}_{11}(t) & \mathbf{B}_{12}(t) & ... & \mathbf{B}_{1n}(t) \\
\mathbf{0} & \mathbf{B}_{22}(t) & ... & \mathbf{B}_{2n}(t) \\
\mathbf{0} & \mathbf{0} & \ddots & \vdots \\
\mathbf{0} & \mathbf{0} & \mathbf{0} & \mathbf{B}_{nn}(t)%
\end{pmatrix}%
,
\end{equation*}%
with%
\begin{eqnarray*}
\mathbf{B}_{ll}(t) &=&e^{\mathbf{A}_{ll}t},l=1,...,n \\
\mathbf{B}_{lj}(t) &=&\int\limits_{0}^{t}\mathbf{M}^{(l,j)}(t,s_{1})ds_{1} \\
&&+\sum\limits_{k=1}^{j-l-1}\int\limits_{0}^{t}\int\limits_{0}^{s_{1}}...%
\int\limits_{0}^{s_{k}}\sum\limits_{l<i_{1}<...<i_{k}<j}\mathbf{M}%
^{(l,i_{1},...,i_{k},j)}(t,s_{1},...,s_{k+1})ds_{k+1}...ds_{1}, \\
l &=&1,...,n-1,\text{ }j=l+1,...,n,
\end{eqnarray*}%
where for any multi-index $(i_{1},...,i_{k})\in
\mathbb{N}
^{k}$ and vector $(s_{1},...,s_{k})\in
\mathbb{R}
^{k}$ the matrices $\mathbf{M}^{(i_{1},...,i_{k})}(s_{1},...,s_{k})$
are
defined by%
\begin{equation*}
\mathbf{M}^{(i_{1},...,i_{k})}(s_{1},...,s_{k})=\left(
\prod\limits_{r=1}^{k-1}e^{\mathbf{A}_{i_{r}i_{r}}(s_{r}-s_{r+1})}\mathbf{A}%
_{i_{r}i_{r+1}}\right) e^{\mathbf{A}_{i_{k}i_{k}}s_{k}}.
\end{equation*}
\end{theorem}

\section{Simplified formulas for the first two moments}

In this section simplified formulas for the first two moments of the
linear SDE (\ref{SS1}) and two of their special forms are derived.
Equations with multiplicative and additive noises as well autonomous
and nonautonomous equations are distinguished.

\subsection{Equations with multiplicative and/or additive noises}

\begin{lemma}
\label{Lemma beta}%
\begin{equation}
vec(\mathcal{B}(s+t_{0}))=\mathcal{B}_{1}+\mathcal{B}_{2}s+\mathcal{B}%
_{3}s^{2}+\mathcal{B}_{4}\mathbf{e}^{\mathbf{C}s}\mathbf{r}+s\mathcal{B}%
_{5}\mathbf{e}^{\mathbf{C}s}\mathbf{r,} \label{vec beta}%
\end{equation}
for all $s\geq0$, where the vectors $\mathcal{B}_{1}=vec(\mathbf{\beta}%
_{1})+\beta_{4}\mathbf{m}_{0}$, $\mathcal{B}_{2}=vec(\mathbf{\beta}%
_{2})+\mathbf{\beta}_{5}\mathbf{m}_{0}$ and
$\mathcal{B}_{3}=vec(\mathbf{\beta }_{3})$, and the matrices
$\mathcal{B}_{4}=\mathbf{\beta}_{4}\mathbf{L}$ and
$\mathcal{B}_{5}=\mathbf{\beta}_{5}\mathbf{L}$ are defined in terms
of the\ matrices
\begin{align*}
\mathbf{\beta}_{1}  &  =\sum\limits_{i=1}^{m}(\mathbf{b}_{i,0}+\mathbf{b}%
_{i,1}t_{0})(\mathbf{b}_{i,0}+\mathbf{b}_{i,1}t_{0})^{\intercal}\\
\mathbf{\beta}_{2}  &  =\sum\limits_{i=1}^{m}(\mathbf{b}_{i,0}+\mathbf{b}%
_{i,1}t_{0})\mathbf{b}_{i,1}^{\intercal}+\mathbf{b}_{i,1}(\mathbf{b}%
_{i,0}+\mathbf{b}_{i,1}t_{0})^{\intercal}\\
\mathbf{\beta}_{3}  &  =\sum\limits_{i=1}^{m}\mathbf{b}_{i,1}\mathbf{b}%
_{i,1}^{\intercal}\\
\mathbf{\beta}_{4}  &  =(\mathbf{a}_{0}+\mathbf{a}_{1}t_{0})\oplus
(\mathbf{a}_{0}+\mathbf{a}_{1}t_{0})+\sum\limits_{i=1}^{m}(\mathbf{b}%
_{i,0}+\mathbf{b}_{i,1}t_{0})\otimes\mathbf{B}_{i}+\mathbf{B}_{i}%
\otimes(\mathbf{b}_{i,0}+\mathbf{b}_{i,1}t_{0})\\
\mathbf{\beta}_{5}  &  =\mathbf{a}_{1}\oplus\mathbf{a}_{1}+\sum\limits_{i=1}%
^{m}\mathbf{b}_{i,1}\otimes\mathbf{B}_{i}+\mathbf{B}_{i}\otimes\mathbf{b}%
_{i,1}.
\end{align*}

\end{lemma}

\begin{proof}
The formula for $vec(\mathcal{B}(s+t_{0}))$ is directly obtained by
substituting (\ref{solucion m}) in (\ref{beta}) with $t=s+t_{0}$.
\end{proof}

The main result of this paper is the following.

\begin{theorem}
\label{Theorem MV1}Let $\mathbf{x}$ be the solution of the linear
SDE (\ref{SS1}) with moments $\mathbf{m}_{0}=E(\mathbf{x}(t_{0}))$
and
$\mathbf{P}_{0}=E(\mathbf{x}(t_{0})\mathbf{x}^{\intercal}(t_{0}))$
at $t_{0}$.
Then, the first two moments of $\mathbf{x}$ can be computed as%
\[
\mathbf{m}_{t}=\mathbf{m}_{0}+\mathbf{L}_{2}e^{\mathbf{M}(t-t_{0})}\mathbf{u}%
\]
and
\[
vec(\mathbf{P}_{t})=\mathbf{L}_{1}e^{\mathbf{M}(t-t_{0})}\mathbf{u}%
\]
for all $t\in\lbrack t_{0},T]$, where the $(d^{2}+2d+7)-$dimensional
vector
$\mathbf{u}$ and the matrices $\mathbf{M}$, $\mathbf{L}_{1}$, $\mathbf{L}_{2}%
$\ are defined as
\[
\mathbf{M}=\left[
\begin{array}
[c]{cccccc}%
\mathcal{A} & \mathcal{B}_{5} & \mathcal{B}_{4} & \mathcal{B}_{3} &
\mathcal{B}_{2} & \mathcal{B}_{1}\\
\mathbf{0} & \mathbf{C} & \mathbf{I}_{d+2} & \mathbf{0} & \mathbf{0}
&
\mathbf{0}\\
\mathbf{0} & \mathbf{0} & \mathbf{C} & \mathbf{0} & \mathbf{0} & \mathbf{0}\\
\mathbf{0} & \mathbf{0} & \mathbf{0} & 0 & 2 & 0\\
\mathbf{0} & \mathbf{0} & \mathbf{0} & 0 & 0 & 1\\
\mathbf{0} & \mathbf{0} & \mathbf{0} & 0 & 0 & 0
\end{array}
\right]  ,\text{ }\mathbf{u}=\left[
\begin{array}
[c]{c}%
vec(\mathbf{P}_{0})\\
\mathbf{0}\\
\mathbf{r}\\
0\\
0\\
1
\end{array}
\right]
\]
and%
\[
\mathbf{L}_{2}=\left[
\begin{array}
[c]{ccc}%
\mathbf{0}_{d\times(d^{2}+d+2)} & \mathbf{I}_{d} & \mathbf{0}_{d\times5}%
\end{array}
\right]  \text{, }\mathbf{L}_{1}=\left[
\begin{array}
[c]{cc}%
\mathbf{I}_{d^{2}} & \mathbf{0}_{d^{2}\times(2d+7)}%
\end{array}
\right]  \text{,}%
\]
with $\mathcal{B}_{i}$ defined as in Lemma \ref{Lemma beta},
$\mathbf{C}$, $\mathbf{r}$ in (\ref{Matrix C noauto}), and
$\mathcal{A}$ in (\ref{Matrix A}).
\end{theorem}

\begin{proof}
From (\ref{solucion P}) and (\ref{vec beta}) follows that
\[
vec(\mathbf{P}_{t})=\mathbf{F}_{1}vec(\mathbf{P}_{0})+\mathbf{K}%
_{1}+\mathbf{H}_{1}\mathbf{r}%
\]
where
\[
\mathbf{F}_{1}=e^{\mathcal{A}(t-t_{0})},
\]%
\[
\mathbf{K}_{1}=%
{\displaystyle\int\limits_{0}^{t-t_{0}}}
e^{\mathcal{A}(t-t_{0}-s)}\mathcal{B}_{1}ds+%
{\displaystyle\int\limits_{0}^{t-t_{0}}}
{\displaystyle\int\limits_{0}^{s}}
e^{\mathcal{A}(t-t_{0}-s)}\mathcal{B}_{2}duds+2%
{\displaystyle\int\limits_{0}^{t-t_{0}}}
{\displaystyle\int\limits_{0}^{s}}
{\displaystyle\int\limits_{0}^{u}}
e^{\mathcal{A}(t-t_{0}-s)}\mathcal{B}_{3}drduds
\]
and%
\[
\mathbf{H}_{1}=%
{\displaystyle\int\limits_{0}^{t-t_{0}}}
e^{\mathcal{A}(t-t_{0}-s)}\mathcal{B}_{4}e^{\mathbf{C}s}ds+%
{\displaystyle\int\limits_{0}^{t-t_{0}}}
{\displaystyle\int\limits_{0}^{s}}
e^{\mathcal{A}(t-t_{0}-s)}\mathcal{B}_{5}e^{\mathbf{C}s}duds.
\]
Further, with $\mathbf{F}_{3}=\mathbf{e}^{\mathbf{C}(t-t_{0})}$,
(\ref{solucion m}) can be written as%
\[
\mathbf{m}_{t}=\mathbf{m}_{0}+\mathbf{LF}_{3}\mathbf{r,}%
\]
where the matrices $\mathbf{L}$, $\mathbf{C}$, $\mathbf{r}$ are
defined as in (\ref{Matrix C noauto}). Thus, by a direct application
of Theorem \ref{Theorem VanLoan} (alternatively, Lemma \ref{Lemma
VanLoan} for $\mathbf{F}_{1}$, $\mathbf{F}_{3}$ and $\mathbf{H}_{1}$
can be used by simplicity) follows that
\[
\left[
\begin{array}
[c]{cccccc}%
\mathbf{F}_{1} & - & \mathbf{H}_{1} & - & - & \mathbf{K}_{1}\\
\mathbf{0} & \mathbf{-} & \mathbf{-} & - & - & -\\
\mathbf{0} & \mathbf{0} & \mathbf{F}_{3} & \mathbf{-} & \mathbf{-} &
\mathbf{0}\\
\mathbf{0} & \mathbf{0} & \mathbf{0} & \mathbf{-} & \mathbf{-} & \mathbf{-}\\
\mathbf{0} & \mathbf{0} & \mathbf{0} & \mathbf{0} & \mathbf{-} & \mathbf{-}\\
\mathbf{0} & \mathbf{0} & \mathbf{0} & \mathbf{0} & \mathbf{0} & \mathbf{-}%
\end{array}
\right]  =e^{\mathbf{M}(t-t_{0})}\text{ with\ }\mathbf{M}=\left[
\begin{array}
[c]{cccccc}%
\mathcal{A} & \mathcal{B}_{5} & \mathcal{B}_{4} & \mathcal{B}_{3} &
\mathcal{B}_{2} & \mathcal{B}_{1}\\
\mathbf{0} & \mathbf{C} & \mathbf{I}_{d+2} & \mathbf{0} & \mathbf{0}
&
\mathbf{0}\\
\mathbf{0} & \mathbf{0} & \mathbf{C} & \mathbf{0} & \mathbf{0} & \mathbf{0}\\
\mathbf{0} & \mathbf{0} & \mathbf{0} & 0 & 2 & 0\\
\mathbf{0} & \mathbf{0} & \mathbf{0} & 0 & 0 & 1\\
\mathbf{0} & \mathbf{0} & \mathbf{0} & 0 & 0 & 0
\end{array}
\right]  ,
\]
where $\mathcal{B}_{i}$ are defined as in Lemma \ref{Lemma beta},
and $\mathcal{A}$ in (\ref{Matrix A}). This implies that
\[
\mathbf{m}_{t}=\mathbf{m}_{0}+\mathbf{Lv}_{2}%
\]
and
\[
vec(\mathbf{P}_{t})=\mathbf{v}_{1}\mathbf{,}%
\]
where the $d^{2}$-dimensional vector $\mathbf{v}_{1}$ and the
$d$-dimensional vector $\mathbf{v}_{2}$ are defined as
\[
\left[
\begin{array}
[c]{c}%
\mathbf{v}_{1}\\
\mathbf{-}\\
\mathbf{v}_{2}\\
-\\
-\\
-
\end{array}
\right]  =e^{\mathbf{M}(t-t_{0})}\mathbf{u}\text{ \ with \
}\mathbf{u}=\left[
\begin{array}
[c]{c}%
vec(\mathbf{P}_{0})\\
\mathbf{0}\\
\mathbf{r}\\
0\\
0\\
1
\end{array}
\right]  .
\]
Proof concludes by verifying that
$\mathbf{Lv}_{2}=\mathbf{L}_{2}\mathbf{v}$ and
$\mathbf{v}_{1}=\mathbf{L}_{1}\mathbf{v}$, where
$\mathbf{v}^{\intercal }=[\mathbf{v}_{1},-,\mathbf{v}_{2},---]$,
$\mathbf{L}_{2}=\left[
\begin{array}
[c]{ccc}%
\mathbf{0}_{d\times(d^{2}+d+2)} & \mathbf{L} & \mathbf{0}_{d\times3}%
\end{array}
\right]  $ and $\mathbf{L}_{1}=\left[
\begin{array}
[c]{cc}%
\mathbf{I}_{d^{2}} & \mathbf{0}_{d^{2}\times(2d+7)}%
\end{array}
\right]  $.
\end{proof}

For autonomous equations with multiplicative and/or additive noises
the formulas of the previous theorem can be simplified as follows.

\begin{theorem}
\label{Theorem MV2}Let $\mathbf{x}$ be the solution of the linear
SDE (\ref{SS1}) with $\mathbf{a}_{1}=\mathbf{b}_{i,1}=\mathbf{0}$.
Let
$\mathbf{m}_{0}=E(\mathbf{x}(t_{0}))$ and $\mathbf{P}_{0}=E(\mathbf{x}%
(t_{0})\mathbf{x}^{\intercal}(t_{0}))$ be moments of $\mathbf{x}$ at
$t_{0}$.
Then, the first two moments of $\mathbf{x}$ can be computed as%
\[
\mathbf{m}_{t}=\mathbf{m}_{0}+\mathbf{L}_{2}e^{\mathbf{M}(t-t_{0})}\mathbf{u}%
\]
and
\[
vec(\mathbf{P}_{t})=\mathbf{L}_{1}e^{\mathbf{M}(t-t_{0})}\mathbf{u}%
\]
for all $t\in\lbrack t_{0},T]$, where the $(d^{2}+d+2)-$dimensional\
vector $\mathbf{u}$ and the matrices $\mathbf{M}_{1}$,
$\mathbf{L}_{1}$ are defined as
\[
\mathbf{M}=\left[
\begin{array}
[c]{ccc}%
\mathcal{A} & \mathcal{B}_{1} & \mathcal{B}_{4}\\
\mathbf{0} & 0 & \mathbf{0}\\
\mathbf{0} & \mathbf{0} & \mathbf{C}%
\end{array}
\right]  \text{, \ \ \ }\mathbf{u}=\left[
\begin{array}
[c]{c}%
vec(\mathbf{P}_{0})\\
1\\
\mathbf{r}%
\end{array}
\right]
\]
and%
\[
\mathbf{L}_{2}=\left[
\begin{array}
[c]{ccc}%
\mathbf{0}_{d\times(d^{2}+1)} & \mathbf{I}_{d} & \mathbf{0}_{d\times1}%
\end{array}
\right]  \text{,\ }\mathbf{L}_{1}=\left[
\begin{array}
[c]{cc}%
\mathbf{I}_{d^{2}} & \mathbf{0}_{d^{2}\times(d+2)}%
\end{array}
\right]
\]
with $\mathcal{B}_{1}$ and $\mathcal{B}_{4}$ defined as in Lemma
\ref{Lemma beta}, $\mathbf{C}$, $\mathbf{r}$ in (\ref{Matrix C
auto}), and $\mathcal{A}$ in (\ref{Matrix A}).
\end{theorem}

\begin{proof}
Since $\mathbf{a}_{1}=\mathbf{b}_{i,1}=\mathbf{0}$,
$\mathcal{B}_{2}=$ $\mathcal{B}_{3}=$ $\mathcal{B}_{5}=\mathbf{0}$.
Thus, from (\ref{solucion P}) and (\ref{vec beta}) follows that
\[
vec(\mathbf{P}_{t})=\mathbf{F}_{1}vec(\mathbf{P}_{0})+\mathbf{G}%
_{1}+\mathbf{H}_{1}\mathbf{r,}%
\]
where
\[
\mathbf{F}_{1}=e^{\mathcal{A}(t-t_{0})},
\]%
\[
\mathbf{G}_{1}=%
{\displaystyle\int\limits_{0}^{t-t_{0}}}
e^{\mathcal{A}(t-t_{0}-s)}\mathcal{B}_{1}ds
\]
and%
\[
\mathbf{H}_{1}=%
{\displaystyle\int\limits_{0}^{t-t_{0}}}
e^{\mathcal{A}(t-t_{0}-s)}\mathcal{B}_{4}e^{\mathbf{C}s}ds.
\]
By a direct application of Lemma \ref{Lemma VanLoan} follows that%
\[
\mathbf{m}_{t}=\mathbf{m}_{0}+\mathbf{LF}_{3}\mathbf{r}%
\]
and%
\[
vec(\mathbf{P}_{t})=\mathbf{F}_{1}vec(\mathbf{P}_{0})+\mathbf{G}%
_{1}+\mathbf{H}_{1}\mathbf{r,}%
\]
where the matrices $\mathbf{F}_{1}$, $\mathbf{F}_{3}$,
$\mathbf{K}_{1}$, and $\mathbf{H}_{1}$ are defined as
\[
\left[
\begin{array}
[c]{ccc}%
\mathbf{F}_{1} & \mathbf{G}_{1} & \mathbf{H}_{1}\\
\mathbf{0} & \mathbf{-} & \mathbf{-}\\
\mathbf{0} & \mathbf{0} & \mathbf{F}_{3}%
\end{array}
\right]  =e^{\mathbf{M}(t-t_{0})}\text{\ with }\mathbf{M}=\left[
\begin{array}
[c]{ccc}%
\mathcal{A} & \mathcal{B}_{1} & \mathcal{B}_{4}\\
\mathbf{0} & 0 & \mathbf{0}\\
\mathbf{0} & \mathbf{0} & \mathbf{C}%
\end{array}
\right]  ,
\]
with $\mathcal{B}_{1}$ and $\mathcal{B}_{4}$ defined as in Lemma
\ref{Lemma beta}, $\mathbf{C}$, $\mathbf{r}$ in (\ref{Matrix C
auto}), and
$\mathcal{A}$ in (\ref{Matrix A}). This implies that%
\[
\mathbf{m}_{t}=\mathbf{m}_{0}+\mathbf{Lv}_{2}%
\]
and%
\[
vec(\mathbf{P}_{t})=\mathbf{v}_{1}\mathbf{,}%
\]
where the $d^{2}$-dimensional vector $\mathbf{v}_{1}$ and the
$d$-dimensional vector $\mathbf{v}_{2}$ are defined as
\[
\left[
\begin{array}
[c]{c}%
\mathbf{v}_{1}\\
\mathbf{-}\\
\mathbf{v}_{2}%
\end{array}
\right]  =e^{\mathbf{M}(t-t_{0})}\mathbf{u}\text{ \ with\
}\mathbf{u}=\left[
\begin{array}
[c]{c}%
vec(\mathbf{P}_{0})\\
1\\
\mathbf{r}%
\end{array}
\right]  .
\]
Proof concludes by verifying that
$\mathbf{Lv}_{2}=\mathbf{L}_{2}\mathbf{v}$ and
$\mathbf{v}_{1}=\mathbf{L}_{1}\mathbf{v}$, where
$\mathbf{v}^{\intercal }=[\mathbf{v}_{1},-,\mathbf{v}_{2}]$,
$\mathbf{L}_{2}=\left[
\begin{array}
[c]{cc}%
\mathbf{0}_{d\times(d^{2}+1)} & \mathbf{L}%
\end{array}
\right]  $ and $\mathbf{L}_{1}=\left[
\begin{array}
[c]{cc}%
\mathbf{I}_{d^{2}} & \mathbf{0}_{d^{2}\times(d+2)}%
\end{array}
\right]  $.
\end{proof}

\subsection{Equations with additive noise}

For autonomous SDEs with additive noise an additional simplification
of the explicit formulas for the first two moments can be archived.

\begin{theorem}
\label{Theorem MV3}Let $\mathbf{x}$ be the solution of the linear
SDE
(\ref{SS1}) with $\mathbf{B}_{i}=\mathbf{0}$ and $\mathbf{a}_{1}%
=\mathbf{b}_{i,1}=\mathbf{0}$. Let
$\mathbf{m}_{0}=E(\mathbf{x}(t_{0}))$ and
$\mathbf{P}_{0}=E(\mathbf{x}(t_{0})\mathbf{x}^{\intercal}(t_{0}))$
be moments of $\mathbf{x}$ at $t_{0}$. Then, the first two moments
of $\mathbf{x}$ can be
computed as%
\[
\mathbf{m}_{t}=\mathbf{m}_{0}+\mathbf{k}_{1}%
\]
and%
\[
\mathbf{P}_{t}=\mathbf{F}_{1}\mathbf{P}_{0}\mathbf{F}_{1}^{\intercal
}+\mathbf{H}_{1}\mathbf{F}_{1}^{\intercal}+\mathbf{F}_{1}\mathbf{H}%
_{1}^{\intercal}%
\]
for all $t\in\lbrack t_{0},T]$, where the matrices $\mathbf{F}_{1}$,
$\mathbf{H}_{1}$, and $\mathbf{k}_{1}$ are defined as
\[
\left[
\begin{array}
[c]{cccc}%
\mathbf{F}_{1} & \mathbf{-} & \mathbf{H}_{1} & \mathbf{k}_{1}\\
\mathbf{0} & \mathbf{-} & \mathbf{-} & -\\
\mathbf{0} & \mathbf{0} & \mathbf{-} & -\\
\mathbf{0} & \mathbf{0} & \mathbf{0} & -
\end{array}
\right]  =e^{\mathbf{M}(t-t_{0})},
\]
being%
\[
\mathbf{M}=\left[
\begin{array}
[c]{cccc}%
\mathbf{A} & \mathbf{a}_{0} & \mathbf{a}_{0}\mathbf{m}_{0}^{\intercal}%
+\frac{1}{2}\sum\limits_{i=1}^{m}\mathbf{b}_{i,0}\mathbf{b}_{i,0}^{\intercal}
& \mathbf{Am}_{0}+\mathbf{a}_{0}\\
\mathbf{0} & 0 & (\mathbf{Am}_{0}+\mathbf{a}_{0})^{\intercal} & 0\\
\mathbf{0} & \mathbf{0} & -\mathbf{A}^{\intercal} & \mathbf{0}\\
\mathbf{0} & \mathbf{0} & \mathbf{0} & 0
\end{array}
\right]  .
\]
a $(2d+2)\times(2d+2)$ matrix.
\end{theorem}

\begin{proof}
Thanks to $\mathbf{B}_{i}=\mathbf{0}$, $\mathbf{a}_{1}=\mathbf{b}%
_{i,1}=\mathbf{0}$\textbf{,} the commutation of the matrices $(\mathbf{I}%
\otimes\mathbf{A})$ and $(\mathbf{A}\otimes\mathbf{I})$, and by
using expressions (\ref{solucion P}) and (\ref{beta}) it is obtained
that
\[
\mathbf{P}_{t}=e^{\mathbf{A}(t-t_{0})}\mathbf{P}_{0}e^{\mathbf{A}^{\intercal
}(t-t_{0})}+%
{\displaystyle\int\limits_{0}^{t-t_{0}}}
e^{\mathbf{A}(t-t_{0}-s)}\mathcal{B}(s+t_{0})e^{\mathbf{A}^{\intercal}%
(t-t_{0}-s)}ds,
\]
where now $\mathcal{B}(s+t_{0})$ reduces to
\[
\mathcal{B}(s+t_{0})=\mathbf{a}_{0}\mathbf{m}_{s+t_{0}}^{\intercal}%
+\mathbf{m}_{s+t_{0}}\mathbf{a}_{0}^{\intercal}+\sum\limits_{i=1}%
^{m}\mathbf{b}_{i,0}\mathbf{b}_{i,0}^{\intercal}.
\]
Taking into account that $\mathbf{m}_{s+t_{0}}$ can be rewritten as
\cite{Jimenez02 SCL}%
\[
\mathbf{m}_{s+t_{0}}=\mathbf{m}_{0}+\mathbf{k}_{1},
\]
with%
\[
\mathbf{k}_{1}=%
{\displaystyle\int\limits_{0}^{s}}
e^{\mathbf{A}(s-u)}(\mathbf{Am}_{0}+\mathbf{a}_{0})du
\]
we have that
\[
\mathbf{a}_{0}\mathbf{m}_{s+t_{0}}^{\intercal}=\mathbf{a}_{0}\mathbf{m}%
_{0}^{\intercal}+\mathbf{a}_{0}(\mathbf{Am}_{0}+\mathbf{a}_{0})^{\intercal}%
{\displaystyle\int\limits_{0}^{s}}
e^{\mathbf{A}^{\intercal}(s-u)}du,
\]
and so $\mathbf{P}_{t}$ can be rewritten as%
\[
\mathbf{P}_{t}=\mathbf{F}_{1}\mathbf{P}_{0}\mathbf{F}_{1}^{\intercal
}+\mathbf{H}_{1}\mathbf{F}_{1}^{\intercal}+\mathbf{F}_{1}\mathbf{H}%
_{1}^{\intercal}\mathbf{,}%
\]
where%
\[
\mathbf{F}_{1}=e^{\mathbf{A}(t-t_{0})}%
\]
and%
\begin{align*}
\mathbf{H}_{1}  & =%
{\displaystyle\int\limits_{0}^{t-t_{0}}}
e^{\mathbf{A}(t-t_{0}-s)}(\mathbf{a}_{0}\mathbf{m}_{0}^{\intercal}+\frac{1}%
{2}\sum\limits_{i=1}^{m}\mathbf{b}_{i,0}\mathbf{b}_{i,0}^{\intercal
}\mathcal{)}e^{-\mathbf{A}^{\intercal}s}ds\\
& +%
{\displaystyle\int\limits_{0}^{t-t_{0}}}
e^{\mathbf{A}(t-t_{0}-s)}\mathbf{a}_{0}(\mathbf{Am}_{0}+\mathbf{a}%
_{0})^{\intercal}%
{\displaystyle\int\limits_{0}^{s}}
e^{-\mathbf{A}^{\intercal}u}duds.
\end{align*}
Proof concludes by a direct application of Lemma \ref{Lemma VanLoan}
\end{proof}

Note that Theorem \ref{Theorem MV3} provides an explicit formula for
the first two moments of autonomous linear equations with additive
noise in terms of just one exponential matrix. This new result
complements both, the well known formulas for the mean and variance
of these equations that can be straightforward obtained from Theorem
1 in \cite{Van Loan 1978} (Lemma \ref{Lemma VanLoan} here) in terms
of two exponential matrices of dimensions $d+1$ and $2d$, and the
formulas for the mean and variance obtained in \cite{Carbonell06
JCAM} in terms of just one exponential matrix of dimension $2d+2$.
At a glance, the difference between the new and previous results
seems to be not so large, but this difference becomes very relevant
when the recursive calculation of the second moment (not the
variance) is required (see, e.g. the forthcoming papers
\cite{Jimenez ALL filter,Jimenez AQML,Jimenez AI}).

\section{Computational issues and numerical simulations}

Theorem \ref{Theorem MV1} provides explicit formulas for the first
two moments of the linear SDE (\ref{SS1}) in terms of an exponential
matrix of dimension $d^{2}+2d+7$. By using the well-known expression
\[
var(\mathbf{z})=E(\mathbf{zz}^{\intercal})-E(\mathbf{z})E(\mathbf{z}%
^{\intercal})
\]
for the variance $var(\mathbf{z})$ of a random variable
$\mathbf{z}$, the variance $var(\mathbf{x})$ of $\mathbf{x}$
solution of (\ref{SS1}) can be straightforward computed as
\[
var(\mathbf{x}(t))=\mathbf{P}_{t}-\mathbf{m}_{t}\mathbf{m}_{t}^{\intercal}%
\]
for all $t\in\lbrack t_{0},T]$, where $\mathbf{m}_{t}$ and
$\mathbf{P}_{t}$ are given as in Theorem \ref{Theorem MV1}.

By taking into account that the explicit formulas for the mean and
variance of (\ref{SS1}) obtained in \cite{Jimenez02 SCL,Jimenez03
IJC} involve the computation of seven exponential matrices of
different dimensions up to a maximum of $3d^{2}+4d+4$, it is obvious
the remarkable benefits of the new simplified formulas. From a
computational viewpoint, this includes a considerable reduction of
the computer storage capacity and the computational time required
for their evaluations through the well known Pad\'{e} method
\cite{Moler03, Jimenez11 BIT} for exponential matrices. But, in
addition, the new formulas allow the efficient computation of the
mean and variance of high dimensional systems of the linear SDEs by
means of the Krylov subspace method \cite{Moler03, Jimenez11 BIT}
for exponential matrices, which is crucial in many practical
situations. Other advantage of the exponential form of these
formulas is the flow property the exponential operator, which allows
an extra reduction of the computational time when the mean and
variance of (\ref{SS1}) are required on consecutive time instants
with multiplicity. In this case, the firsts two conditional moments
at the first time instant after the initial one is computed through
the exponential matrix of Theorem \ref{Theorem MV1}, whereas the
others at the remainder times are computed by simple multiplications
of the exponential matrix just mentioned.

Theorems \ref{Theorem MV2} and \ref{Theorem MV3} provide explicit
formulas for the first two moments of autonomous linear SDEs, which
involve an exponential matrix of lower dimensionality: $d^{2}+d+2$
for equation with multiplicative noise, and $2d+2$ for equations
with additive noise. This yields extra advantages in a number of
important applications.\newline%

\begin{tabular}
[c]{|c|c|c|}\hline Type of SDE / dimension & $2$ & $8$\\\hline
\multicolumn{1}{|l|}{Non Autonomous, with Multiplicative Noise} &
\multicolumn{1}{|l|}{$0.457$} &
\multicolumn{1}{|l|}{$0.036$}\\\hline
\multicolumn{1}{|l|}{Autonomous, with Multiplicative Noise} & \multicolumn{1}{|l|}{$0.322$%
} & \multicolumn{1}{|l|}{$0.021$}\\\hline
\multicolumn{1}{|l|}{Autonomous, with Additive Noise} &
\multicolumn{1}{|l|}{$0.072$} &
\multicolumn{1}{|l|}{$0.001$}\\\hline
\end{tabular}

{\small Table I: Relative computational time between the new and old
formulas for the moments of linear SDEs.}\newline

As illustration, the performance of the new and old formulas are
compared for three types of linear equations. In particular, the
equations
\begin{align*}
d\mathbf{x}(t) &  =(-\mathbf{Hx}(t)+\mathbf{1}t)dt+\mathbf{Hx}(t)d\mathbf{w}%
(t),\\
d\mathbf{x}(t) & =-\mathbf{Hx}(t)dt+\mathbf{Hx}(t)d\mathbf{w}(t),
\end{align*}
and%
\[
d\mathbf{x}(t)=-\mathbf{Hx}(t)dt+\mathbf{1}d\mathbf{w}(t)
\]
with $t\in\lbrack0,1]$ and initial conditions
$E(\mathbf{x}(0))=\mathbf{1}$,
$E(\mathbf{x}(0)\mathbf{x}^{\intercal}(0))=\mathbf{11}^{\intercal}$
were considered, where $\mathbf{H}$ denotes the $d\times d$ Hilbert
matrix and $\mathbf{1}$ the $d$-dimensional unit vector. The
formulas of the Theorems \ref{Theorem MV1}, \ref{Theorem MV2} and
\ref{Theorem MV3} were used to compute the moments of the first,
second and third equation, respectively, at $t=1$. With the same
purpose, the formulas of Theorem 3 in \cite{Jimenez03 IJC} were used
for the three SDEs. For equations with dimensions $d=2$ and $d=8$,
Table I presents the relative computational time between the new and
old formulas, which is calculated as the ratio of the CPU time
consumed for these formulas in each equation. Observe as, in all the
cases, the new formulas exhibit a substantial reduction of the
computational cost. As it was expected, this reduction clearly
increases with the dimensionality and the simplicity of the
equation.

Finally, it is worth noting that the simplified formulas derived
here have allowed a computationally efficient implementation of the
approximate filters and estimators recently proposed in
\cite{Jimenez ALL filter,Jimenez AQML,Jimenez AI} for the
identification of diffusion processes from a reduced number of
discrete observations distant in time.

\section{Conclusions}

In this paper, explicit formulas for the mean and variance of linear
stochastic differential equations were derived in terms of an
exponential matrix. With respect to the formulas proposed in a
previous paper the new ones have a number of clear advantages: 1)
they involve the computation of just one exponential matrix of lower
dimensionality; 2) for high dimensional SDEs, they can be
straightforward computed though the Krylov subspace method; 3) for
consecutive time instants with multiplicity, their flow property can
be used; and 4) they reduces to simpler forms for autonomous SDEs
and for equations with additive noise. From numerical viewpoint,
this implies a significant reduction of the computer storage
capacity and the computational
time.\\[0.25in]

\textbf{Acknowledgement: }The author thanks to Prof. A. Yoshimoto
for his invitation to the Institute of Statistical Mathematics,
Japan, where the manuscript was completed.



\bibliographystyle{elsarticle-num}



\end{document}